\documentclass[preprint,11pt]{elsarticle}
\usepackage{graphicx} 
\usepackage{multicol}
\usepackage{import}
\usepackage{amssymb}
\usepackage{amsmath}
\usepackage{amsthm}
\usepackage{tikz}
\usepackage{tikz-network}
\usepackage{setspace}
\usepackage{realhats}
\usepackage{subfig}
\usepackage{enumerate}
\usepackage[margin=2.9cm]{geometry}
\usepackage{hyperref}
\usepackage{multicol}
\usepackage[inline]{showlabels}       

\theoremstyle{plain}
\newtheorem{thm}{Theorem}[section]
\newtheorem{lem}[thm]{Lemma}

\newtheorem{cor}[thm]{Corollary}

\theoremstyle{definition}
\newdefinition{definition}[thm]{Definition}
\newdefinition{ex}[thm]{Example}
\newdefinition{question}[thm]{Question}

\journal{Discrete Mathematics}

\begin{document}

\begin{frontmatter}

\title{Structural properties of graphs and the Universal Difference Property}

\author[1]{Katie Anders\corref{cor1}}
\fnref{fn1}
\ead{kanders@uttyler.edu}
\author[2]{Able Martinez\fnref{fn2}}
\ead{aamartinez3@mavs.coloradomesa.edu}
\author[3]{Patrick McHugh\fnref{fn3}}
\ead{pqm@andrew.cmu.edu}
\author[4]{Jenna Rogers\fnref{fn4}}
\ead{vrogers2@patriots.uttyler.edu}
\author[5]{Remi Salinas Schmeis\fnref{fn5}}
\ead{remis@nmsu.edu}
\cortext[cor1]{corresponding author.}
\affiliation[1]{organization={Department of Mathematics, University of Texas at Tyler},
addressline={3900 University Blvd.},
city={Tyler},
state={TX},
postcode={75707},
country={USA}}
\affiliation[2]{organization={Department of Mathematics and Statistics, Colorado Mesa University},
addressline={1100 North Ave.},
city={Grand Junction},
state={CO},
postcode={81501},
country={USA}}
\affiliation[3]{organization={Department of Mathematical Sciences, Carnegie Mellon University},
addressline={Wean Hall 6113},
city={Pittsburgh},
state={PA},
postcode={15213},
country={USA}}
\affiliation[4]{organization={Department of Mathematics, University of Texas at Tyler},
addressline={3900 University Blvd.},
city={Tyler},
state={TX},
postcode={75707},
country={USA}}
\affiliation[5]{organization={Department of Mathematical Sciences, New Mexico State University},
addressline={1290 Frenger Mall MSC 3MB / Science Hall 236},
city={Las Cruces},
state={NM},
postcode={88003-8001},
country={USA}}

\begin{abstract}

\noindent We study the Universal Difference Property (UDP) introduced by Altınok, Anders, Arreola, Asencio, Ireland, Sarıoğlan, and Smith, focusing on the relationship between the structural properties of a graph and UDP. We present condtions for when UDP must hold on unicyclic graphs. We then prove that if UDP does not hold on an edge-labeled graph, then it cannot hold on any subdivision of that graph.  Additionally, we show that if an edge-labeled graph satisfies the pairwise edge-disjoint path property, then the graph satisfies UDP. Lastly, we explore the relationship between UDP and subgraphs and prove that trees and cycles are the only two families of connected graphs for which UDP must hold for any edge-labeling over any ring.

\end{abstract}


\end{frontmatter}


\section{Introduction}
\label{intro}


Splines have been defined on many different types of structures in various areas of mathematics, with these definitions having numerical, topological, and algebraic interpretations and applications.  Gilbert, Tymoczko, and Viel \cite{Gilbert_2016} defined generalized graph splines, choosing their definition to encompass what was common to splines as they had been used across these different areas of math.  We state their definition here as Definition \ref{def:spline} but must first give the definition of an edge-labeling on a graph.

\begin{definition}
    Given a graph $G = (V, E)$ and a commutative ring $R$ with unity, an \textit{edge-labeling} of $G$ is a function $\alpha: E \to \mathcal{I}(R)$, where $\mathcal{I}(R)$ is the set of all ideals in $R$. The pair $(G, \alpha)$ denotes a graph $G$ with edge-labeling $\alpha$.
\end{definition}

In keeping with this definition of edge-labeled graphs, we restrict our attention in this paper to commutative rings with unity.  Additionally, all graphs are assumed to be connected. 

\begin{definition}\label{def:spline}
    A \textit{generalized spline} on $(G, \alpha)$ over the ring $R$ is a function $\rho: V \to R$ such that for each edge $uv$, the difference $\rho(u) - \rho(v)$ is an element of $\alpha(uv)$.
\end{definition}

In this paper, when the context is clear, a generalized spline on an edge-labeled graph will be referred to simply as a spline. Further, given an edge-labeled graph, the ring $R$ from which the edge-labeling comes will be referred to as the base ring for the edge-labeled graph. 

\begin{ex} Figure \ref{splineex} illustrates a spline on an edge-labeled graph $(G,\alpha)$ with base ring $R=\mathbb{Z}[x]$. 
\end{ex}

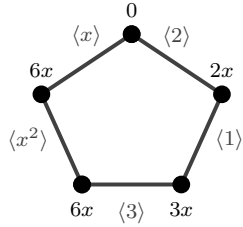
\begin{figure}[h! bt]
    \centering
    \begin{tikzpicture}
        \Vertex[size = 0.2, color = black, label = 0, position = above]{1}
        \Vertex[size = 0.2, color = black]{6}
        \Vertex[size = 0.2, color = black, x = -1.2, y = -0.8, label = $6x$, position = above]{2}
        \Vertex[size = 0.2, color = black, x = -0.66, y = -2, label = $6x$, position = below]{3}
        \Vertex[size = 0.2, color = black, x = 0.66, y = -2, label = $3x$, position = below]{4}
        \Vertex[size = 0.2, color = black, x = 1.2, y = -0.8, label = $2x$, position = above]{5}
        \Vertex[size = 0.2, color = black, x = -0.66, y = -2]{7}

        \Edge[label=$\langle x \rangle$, position={above=1mm}](1)(2)
        \Edge[label=$\langle x^2 \rangle$, position={left=1mm}](2)(3)
        \Edge[label=$\langle 3 \rangle$, position={below=1mm}](3)(4)
        \Edge[label=$\langle 1 \rangle$, position={right=1mm}](4)(5)
        \Edge[label=$\langle 2 \rangle$, position={above=1mm}](5)(1)
        
    \end{tikzpicture}
    \caption{Example of a spline over $\mathbb{Z}[x]$}
    \label{splineex}
\end{figure}

After introducing the notion of a generalized graph spline, Gilbert, Tymoczko, and Viel \cite{Gilbert_2016} proved many foundational results about graph splines and their basic properties.  In particular, they showed that given an edge-labeled graph $(G,\alpha)$ over a base ring $R$, the set of all splines on $(G,\alpha)$ is both a ring and an $R$-module.  Conditions for when the ring of generalized splines is a free $R$-module have been studied, and a major focal point in the investigation of generalized splines is the question of finding a basis for the space of splines on a given edge-labeled graph when the module of splines is free.

  Some authors narrowed their focus by fixing both the type of graph and the base ring.  For example, Handschy, Melnick, and Reinders \cite{handschy2014integergeneralizedsplinescycles} considered cycles over the integers, as did Bowden, Hagen, King, and Reinders \cite{bowden2015basesstructureconstantsgeneralized}, while Mahdavi \cite{mahdavi2016integer} investigated diamond graphs over the integers.  Other authors restrict only one of these, either the base ring or the type of graph, but not both.  Specifically, Gilbert, Tymoczko, and Viel \cite{Gilbert_2016} found a basis for the space of splines on an edge-labeled tree. In contrast, Rose and Suzuki \cite{rose2023generalized} fixed the base ring $\mathbb{Z}$, while Bowden and Tymoczko \cite{bowden2015splinesmodm} initiated work over $\mathbb{Z}/m\mathbb{Z}$.  Philbin, Swift, Tammaro, and Williams \cite{philbin2017splinesintegerquotientrings} also considered the base ring $\mathbb{Z}/m\mathbb{Z}$ and determined a process for finding a minimum generating set for splines on an edge-labeled graph over this base ring.  Additional work on finding a basis for the space of splines has been done in \cite{spline_modules} and \cite{Altinok_basis_criteria_generalized_spline_module}.   Another question of interest is whether an edge-labeled graph admits only constant splines.  Anders, Crans, Foster-Greenwood, Mellor, and Tymoczko investigated this in \cite{Anders_2020} for a class of rings including $\mathbb{Z}/m\mathbb{Z}$.


Applications of graph splines include studying piecewise polynomials on a polyhedral complex \cite{Brion1996}, as well as equivariant cohomology rings of toric varieties \cite{payne2005equivariantchowcohomologytoric}. In the latter area, Goresky, Kottwitz, and MacPherson developed GKM theory, which can be used to study algebraic varieties equipped with a torus action \cite{article}. 

 Throughout this paper, we use $\langle r \rangle$  to denote the principal ideal generated by the element $r$ of the base ring $R$, and we use $\mathcal{P}_G(u,v)$ to denote the set of all paths in the graph $G$ between two vertices $u$ and $v$ in $V(G)$. When only one graph is currently under consideration and the context is clear, we may use $\mathcal{P}(u,v)$ rather than $\mathcal{P}_G(u,v)$.

We investigate a particular property of splines on edge-labeled graphs, which is known as the Universal Difference Property (UDP).  The first step toward the definition of UDP comes from Theorem \ref{containment}.

\begin{thm}\label{containment}\cite{notes}
    Suppose $u, w \in V(G)$ for a graph $(G, \alpha)$ and $P = \langle u, v_1, \dots, v_n, w \rangle$ is a path from $u$ to $w$. Let $\alpha(P) = \alpha(uv_1) + \alpha(v_1v_2) + \dots + \alpha(v_nw)$. If $\rho$ is a spline on $(G, \alpha)$, then $\rho(u) - \rho(w) \in \alpha(P)$. Moreover, if $P_1, \dots , P_m$ are paths from $u$ to $w$, then $\rho(u) - \rho(w) \in \bigcap_{i \in [m]} \alpha(P_i)$.
\end{thm}


Note that Theorem \ref{containment} starts with a spline and moves toward an element of the intersection.  In search of a type of converse to Theorem \ref{containment}, the authors of \cite{notes} asked Question \ref{big_question}, which moves from an element of the intersection toward a spline.

\begin{question}\label{big_question}

Let $(G, \alpha)$ be an edge-labeled graph. Let $u, w$ be connected vertices in $(G, \alpha)$. Let $P_1, P_2, ..., P_m$ be the paths in $G$ from $u$ to $w$. For each $i\in[m]$, let $\alpha(P_i)$ be the sum of the ideals labeling the edges in the path $P_i$. Let $x\in\bigcap_{i\in[m]}{\alpha(P_i)}$. Under what conditions does there exists a spline $\rho$ on $(G, \alpha)$ such that $\rho(u)-\rho(w)=x$?
    
\end{question}

Altınok, Anders, Arreola, Asencio, Ireland, Sarıoğlan, and Smith addressed Question \ref{big_question} in \cite{UDP}, defining the Universal Difference Property.

\begin{definition}\label{def:UDP}
    Let $R$ be a ring and $(G, \alpha)$ an edge-labeled graph. We say that $(G, \alpha)$ satisfies the \textit{Universal Difference Property} (UDP) if for every pair of vertices $u$ and $w$ and every $x \in \bigcap_{i \in [m]} \alpha(P_i)$, there exists a spline $\rho$ on $(G, \alpha)$ such that $\rho(u) - \rho(w) = x$.
\end{definition}

Note that UDP allows us to compare labels assigned by a spline to vertices in the same connected component of the graph, not just adjacent vertices.  Altınok, Anders, Arreola, Asencio, Ireland, Sarıoğlan, and Smith proved the following theorems regarding UDP.

\begin{thm}\label{UPDTreeThm}\cite[Theorem 2.4]{UDP}
    If $G$ is a tree, the Universal Difference Property holds on $(G, \alpha)$ for any edge-labeling $\alpha$ over any base ring $R$.
\end{thm}

\begin{thm}\label{UDPCycleThm}\cite[Theorem 2.6]{UDP}
    If $G$ is a cycle, the Universal Difference Property holds on $(G, \alpha)$ for any edge-labeling $\alpha$ over any base ring $R$.
\end{thm}

\begin{thm}\label{paste}\cite[Theorem 3.5]{UDP}
    Let $G_1$ and $G_2$ be connected graphs such that the Universal Difference Property holds for each and $V(G_1) \cap V(G_2) = \{z\}$. Let $G = G_1 \cup G_2$, where $G$ is created by pasting $G_1$ and $G_2$ at $z$. Then $G$ satisfies the Universal Difference Property if and only if for all $u \in V(G_1)$, $w \in V(G_2)$, we have that 
    \begin{equation}\label{eq:key pasted graphs equation}
        \bigcap \limits_{P \in \mathcal{P}_G(u,w)} \alpha(P) = \left( \bigcap \limits_{P \in \mathcal{P}_G(u,z)} \alpha(P) \right) + \left( \bigcap \limits_{P \in \mathcal{P}_G(z,w)} \alpha(P) \right)
    \end{equation}
\end{thm}

Next we define a Prüfer domain, the notion of which is essential for the main result of \cite{UDP}.  There are many equivalent ways to define a Prüfer domain, as shown in \cite{Jensen paper}, but we give here the same definition used in \cite{UDP}.

\begin{definition}\label{def:Prufer domain}
    Let $R$ be an integral domain. Then $R$ is a \textit{Prüfer domain} if and only if for any nonzero ideals $I$, $J$, and $K$ in $R$ we have $I \cap (J + K) = (I \cap J) + (I \cap K)$. 
\end{definition}

We are now prepared to state the main result of \cite{UDP}.

\begin{thm}\label{intersection_over_sum}\cite[Theorem 4.4]{UDP}
    Let $R$ be an integral domain. Every edge-labeled graph over $R$ satisfies UDP if and only if $R$ is a Prüfer domain.
\end{thm}

The existence of a commutative ring with unity having three nonzero ideals for which the equation in Definition \ref{def:Prufer domain} does not hold is essential to proving several of our theorems. Many examples of such rings exist, including any integral domain that is not a Prüfer domain.  Specifically, the polynomial ring $\mathbb{Z}[x]$ is an integral domain but is not a Prüfer domain.

In Section \ref{unicyclic}, we build upon previous work to find conditions under which UDP holds on a unicyclic graph. In Section \ref{extensions}, we examine subdivisions and multigraphs, proving that if a graph does not satisfy UDP then no subdivision of the graph satisfies UDP, and constructing an edge-labeling on the diamond graph such that UDP does not hold.  See Figure \ref{DG}(c). In Section \ref{ped}, we explore the Pairwise Edge Disjoint Path Property of edge-labeled graphs and show that any graph having this property must satisfy UDP for any edge-labeling over any ring.  Finally, in Section \ref{sec: subgraphs}, we follow up on the work done in \cite{UDP} by providing a complete characterization of the graphs for which UDP must hold for any edge-labeling over any ring.


\section{Unicyclic Graphs and UDP} 
\label{unicyclic}

In this section, we examine the conditions under which UDP must hold on edge-labeled unicyclic graphs.  These conditions are outlined in Theorem \ref{thm: full unicyclic}, the main result of this section.  We begin with an explanation as to why we chose to investigate UDP on edge-labeled unicyclic graphs.  

In \cite{UDP}, Altınok et al. showed that UDP must hold on any edge-labeled tree or cycle over any ring $R$. Theorem \ref{paste} gives a necessary and sufficient condition for UDP to hold on a graph $G$ created by pasting two graphs $G_1$ and $G_2$ at a single vertex, provided $G_1$ and $G_2$ both satisfy UDP.  Because we know that trees and cycles satisfy UDP, we wanted to consider pasting trees and cycles together to form unicyclic graphs and then use Theorem \ref{paste} to determine when such unicyclic graphs would satisfy UDP.

\begin{definition}
    A \textit{unicyclic graph} is a connected graph that contains exactly one cycle.
\end{definition}

Because any unicyclic graph can be formed by pasting trees onto vertices of a cycle, it seems a natural guess that UDP must hold on any unicyclic graph over any ring $R$.  However, we found the following counterexample to that conjecture.
\begin{ex}\label{ex:unicyclic graph not UDP}
    Let $R$ be a ring with nonzero ideals $I, J, K$ such that $I+(J\cap K) \neq (I+J)\cap(I+K)$. Then by Theorem \ref{paste}, it follows that UDP does not hold on the edge-labeled unicyclic graph in Figure \ref{unicyclenoUDP}.
\end{ex}

 \begin{figure}[h! bt]
\centering
\begin{tikzpicture}
    \Vertex[size=0.2, color=black]{1}
    \Vertex[size=0.2, color=black, x=-1, y=-1, label=z, position=above, fontcolor=red]{2}
    \Vertex[size=0.2, color=black, y=-2]{3}
    \Vertex[size=0.2, color=black, x=1, y=-1, label=w, position=above, fontcolor=red]{4}
    \Vertex[size=0.2, color=black, x=-3, y=-1, label=u, position=above, fontcolor=red]{5}
    \Edge[label=I, position={above=1mm}, color=black](2)(5);
    \Edge[label=J, position={above=1mm}, color=black](1)(2);
    \Edge[label=J, position={above=1mm}, color=black](1)(4);
    \Edge[label=K, position={below=1mm}, color=black](2)(3);
    \Edge[label=K, position={below=1mm}, color=black](3)(4);
\end{tikzpicture}
\caption{An edge-labeled unicyclic graph for which UDP does not hold}
\label{unicyclenoUDP}
\end{figure}
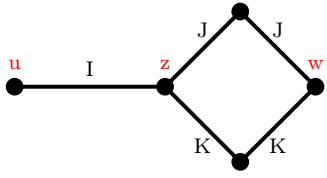

Now that we have seen an example of an edge-labeled unicyclic graph on which UDP does not hold, we pose the following question.

\begin{question}\label{question: general unicyclic}
Under what conditions must UDP hold on an edge-labeled unicyclic graph? 
\end{question}

We will answer Question \ref{question: general unicyclic} in Theorem \ref{thm: full unicyclic}, but first we consider the answer for a specific type of unicyclic graph, one formed by pasting two trees onto a cycle at distinct vertices of the cycle. We begin with a definition from \cite{Gilbert_2016}.

\begin{definition}
Given an edge-labeled graph $(G, \alpha)$, a \textit{edge-labeled subgraph} of $(G, \alpha)$ is an edge-labeled graph $(H, \alpha')$, where $H$ is a subgraph of $G$ and $\alpha'$ is an edge-labeling such that for any $e\in{E(H)}$, $\alpha'(e)=\alpha(e)$.

\end{definition}

We now consider the relationship between UDP on an edge-labeled graph and UDP on a certain type of edge-labeled subgraph.  Note that the interplay between UDP and more general edge-labeled subgraphs will be explored in Section \ref{sec: subgraphs}.

\begin{thm}\label{sub}
    Let $(G, \alpha)$ be an edge-labeled graph and $(G',\alpha')$ an edge-labeled subgraph of $(G,\alpha)$ such that for any $u,v \in V(G')$ we have $\mathcal{P}_{G'}(u,v) = \mathcal{P}_G(u,v)$. If UDP holds on $(G,\alpha)$, then UDP holds on $(G',\alpha')$.
\end{thm}

\begin{proof}
    Let $(G,\alpha)$ and $(G',\alpha')$ be as in the statement of the theorem and assume that UDP holds on $(G,\alpha)$.  Let $a,b \in V(G')$. We have that $\bigcap_{P \in\mathcal{P}_{G(a,b)}}\alpha(P) = \bigcap_{P \in\mathcal{P}_{G'(a,b)}}\alpha '(P)$. Let $x \in \bigcap_{P \in\mathcal{P}_{G'(a,b)}}\alpha '(P)$. Then $x \in \bigcap_{P \in\mathcal{P}_{G(a,b)}}\alpha(P)$. Since UDP is satisfied on $(G,\alpha)$, there exists a spline $\rho: V(G)\rightarrow R$ such that $\rho(a) - \rho(b) = x$. Define $\rho':V(G')\rightarrow R$ by $\rho'(w)=\rho(w)$ for all $w \in V(G')$. Let $z_1$ and $z_2$ be arbitrary adjacent vertices in $G'$. Then $z_1$ and $z_2$ must be adjacent vertices in $G$ because $G'$ is a subgraph of $G$.  Since $\rho$ is a spline on $G$, we have that $\rho'(z_1)-\rho'(z_2)=\rho(z_1)-\rho(z_2)\in\alpha(z_1z_2)=\alpha'(z_1z_2)$, so $\rho'$ is a spline on $G'$. Furthermore, $\rho'(a) - \rho'(b) = \rho(a) - \rho(b) = x$. Thus UDP holds on $(G',\alpha')$.
\end{proof}

We now illustrate a connection between the condition in the hypothesis of Theorem \ref{sub} and the graphs considered in Theorem \ref{paste}.

\begin{lem}\label{path}
Let $G_1$ and $G_2$ be graphs such that $V(G_1) \cap V(G_2) = \{z\}$ and let $G = G_1 \cup G_2$, where $G$ is created by pasting $G_1$ and $G_2$ at $z$. Then for any $u,v \in V(G_1)$, $\mathcal{P}_{G_1}(u,v) = \mathcal{P}_{G}(u,v)$.
\end{lem}

\begin{proof}
    Let $u,v\in V(G_1)$.  Since any path in $G_1$ from $u$ to $v$ is also a path in $G$ from $u$ to $v$, we have $\mathcal{P}_{G_1}(u,v) \subseteq \mathcal{P}_{G}(u,v)$. Now suppose that $\mathcal{P}_{G}(u,v) \nsubseteq \mathcal{P}_{G_1}(u,v)$. Then there is some path $P \in \mathcal{P}_{G}(u,v)$ such that $P \not \in \mathcal{P}_{G_1}(u,v),$ so $P$ must contain a least one vertex $w$ such that $w\in V(G)$ but $w\notin V(G_1)$. Using the construction of $G$ and the fact that $P$ is a path between two vertices in $G_1$, we see that $P$ must visit the vertex $z$ twice. This is a contradiction to $P$ being a path, so it must be that $\mathcal{P}_{G}(u,v) \subseteq \mathcal{P}_{G_1}(u,v)$.

\end{proof}

We are now prepared to answer Question \ref{question: general unicyclic} for the case in which the unicyclic graph is created by pasting two trees onto a cycle at distinct vertices of the cycle.

\begin{thm}\label{cycle2tree}
    Let $C$ be an edge-labeled cycle on $n$ vertices $v_1, v_2, \dots, v_n$. Let $1 \leq i < j \leq n.$ Let $T_i$ and $T_j$ be edge-labeled trees. Let $G_1$ be the graph formed by pasting $T_i$ to $C$ at the vertex $v_i$ and let $G_2$ be the graph formed by pasting $T_j$ to $G_1$ at the vertex $v_j$. Let $\alpha$ be the edge-labeling on $G_2$ such that each edge in $G_2$ has the same label it had in $C$, in $T_i$, or in $T_j$. Consider the following conditions: 
\begin{enumerate}[(i)]
    \item for every $u \in V(C)$ and $w \in V(T_i)$,
    \begin{align*}
        \bigcap_{P\in\mathcal{P}_{G_1}(u,w)} \alpha(P) = \left(\bigcap_{P\in\mathcal{P}_{C}(u,v_i)} \alpha(P)\right) + \left(\bigcap_{P\in\mathcal{P}_{T_i}(v_i,w)} \alpha(P)\right)
    \end{align*}
    \item for every $r \in V(G_1)$ and $s \in V(T_j)$,
    \begin{align*}
        \bigcap_{P\in\mathcal{P}_{G_2}(r,s)} \alpha(P) = \left(\bigcap_{P\in\mathcal{P}_{G_1}(r,v_j)} \alpha(P)\right) + \left(\bigcap_{P\in\mathcal{P}_{T_j}(v_j,s)} \alpha(P)\right)
    \end{align*}
\end{enumerate}
Both conditions (i) and (ii) hold if and only if UDP holds on $G_2$.
\end{thm}

\begin{proof}

  Let the cycle, trees, vertices, graphs, and edge-labeling $\alpha$ be as in the statement of the theorem.

    For the forward direction, assume (i) and (ii) hold. Note that $V(C) \cap V(T_i) = \{v_i\}$.  Since $G_1$ is created by pasting $C$ and $T_i$ at $v_i$ and Condition (i) holds, we can apply Theorem \ref{paste} to conclude that UDP holds on $G_1$. Similarly, $V(G_1) \cap V(T_j)= \{v_j\}$, $G_2$ is created by pasting $T_j$ to $G_1$ at $v_j$ and Condition (ii) holds, so we can again apply Theorem \ref{paste}, this time to conclude that UDP holds on $G_2$.

    For the backward direction, assume UDP holds on $G_2$. For the sake of contradiction, assume Condition (i) does not hold. Then there exists $u \in V(C)$ and $w \in V(T_i)$ such that 
    \begin{align*}    
    \bigcap_{P\in\mathcal{P}_{G_1}(u,w)} \alpha(P) \not = \left(\bigcap_{P\in\mathcal{P}_{C}(u,v_i)} \alpha(P)\right) + \left(\bigcap_{P\in\mathcal{P}_{T_i}(v_i,w)} \alpha(P)\right).
    \end{align*}
    By Theorem \ref{paste}, we can conclude that UDP does not hold on $G_1$. Note that $G_2=G_1 \cup T_j$ with $V(G_1)\cap{V(T_j)}=\{v_j\}$. By Lemma \ref{path}, for any $a, b\in{V(G_1)}$, we have $\mathcal{P}_{G_1}(a, b)=\mathcal{P}_{G_2}(a, b)$. We see that the hypotheses of Theorem \ref{sub} are satisfied, with $G_2$ as the edge-labeled graph and $G_1$ as the subgraph. Since UDP does not hold on $G_1$, the contrapositive of Theorem \ref{sub} tells us that UDP does not hold on $G_2$. This is the desired contradiction, and we conclude that Condition (i) must hold.  Again, since $G_1$ is created by pasting $C$ and $T_i$ at $v_i$ and Condition (i) holds, we can apply Theorem \ref{paste} to conclude that UDP holds on $G_1$.

    It remains to show that Condition (ii) must hold. Notice that $G_2$ is formed by pasting together the graphs $G_1$ and $T_j$ at $v_j$, and $G_1$ and $T_j$ both satisfy UDP. Since UDP holds on $G_2$, Theorem \ref{paste} implies that Condition (ii) must hold.

\end{proof}

\begin{ex}
Let $I, J, K, L$ be ideals of a ring $R$. Let $C$ be a $4$-cycle with vertices $v_1, v_2, v_3, v_4$.  Let $u$ be a vertex distinct from all $v_i$ and $T_2$ be the tree consisting of the single edge $uv_2$.  Let $w$ be a vertex distinct from all $v_i$ and also from $u$, and let $T_4$ be the tree consisting of the single edge $wv_4$.  Let $G_1$ be formed by pasting $T_2$ to $C$ at $v_2$ and $G_2$ be formed by pasting $T_4$ to $G_1$ at $v_4$. Consider the edge labeling on $G_2$ given in Figure \ref{unicycleUDP}. By inspection, we see that the relevant equalities corresponding to Conditions (i) and (ii) in Theorem \ref{cycle2tree} hold.  Thus we can apply Theorem \ref{cycle2tree} to conclude that UDP holds on $G_2$.
\end{ex}

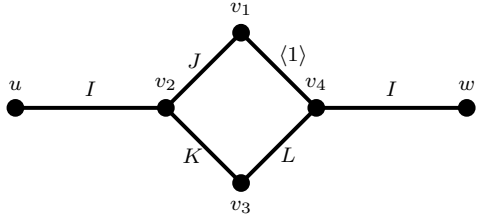
\begin{figure}[h! bt]
\centering
\begin{tikzpicture}
    \Vertex[size=0.2, color=black, label=$v_1$, position=above]{1}
    \Vertex[size=0.2, color=black, x=-1, y=-1, label=$v_2$, position=above]{2}
    \Vertex[size=0.2, color=black, y=-2, label=$v_3$, position=below]{3}
    \Vertex[size=0.2, color=black, x=1, y=-1, label=$v_4$, position=above]{4}
    \Vertex[size=0.2, color=black, x=-3, y=-1, label=$u$, position=above]{5}
    \Vertex[size=0.2, color=black, x=3, y=-1, label=$w$, position=above]{6}
    \Edge[color=black, label=$I$, position={above=1mm}](5)(2);
    \Edge[color=black, label=$J$, position={above left}](1)(2);
    \Edge[color=black, label=$K$, position={below left}](3)(2);
    \Edge[color=black, label=$L$, position={below right}](3)(4);
    \Edge[color=black, label=$I$, position={above=1mm}](4)(6);
    \Edge[color=black, label=$\langle 1 \rangle$, position={above right}](1)(4);
\end{tikzpicture}
\caption{Example of a unicyclic graph satisfying UDP formed by pasting two trees onto a cycle at distinct vertices}
\label{unicycleUDP}
\end{figure}

Theorem \ref{cycle2tree} provides necessary and sufficient conditions for UDP to hold on a unicyclic graph formed by pasting two trees onto a cycle at distinct vertices. The following theorem generalizes this result to show when UDP must hold on a unicyclic graph formed by pasting up to $n$ trees onto a cycle of $n$ vertices.  Since any unicyclic graph can be formed in this way, Theorem \ref{thm: full unicyclic} provides the answer to Question \ref{question: general unicyclic} for any edge-labeled unicyclic graph.
\begin{thm}\label{thm: full unicyclic}
    Let $C$ be a cycle on n vertices $v_1, v_2,...,v_n$. For each 1 $\leq i \leq n$, let $T_i$ be a tree, possibly consisting of a single vertex. Let $G_1$ be the graph formed by pasting $T_1$ to $C$ at the vertex $v_1$. For each 2 $\leq i \leq n$, let $G_i$ be the graph formed by pasting $T_i$ to $G_{i-1}$ at the vertex $v_i$. Then UDP holds on $G_n$ if and only if both of the following conditions hold:
\begin{enumerate}[(i)]
    \item[(i)] UDP holds on $G_{j}$ for all $ 1 \leq j \leq n-1$
    \item[(ii)] for every $u\in V(G_{n-1})$ and $w\in V(T_n)$, $\bigcap\limits_{P\in\mathcal{P}(u,w)}\alpha(P)=\left(\bigcap\limits_{P\in\mathcal{P}(u,v_n)}\alpha(P)\right) + \left(\bigcap\limits_{P\in\mathcal{P}(v_n,w)}\alpha(P)\right)$
\end{enumerate}
\end{thm}

\begin{proof}
    We prove the forward direction first. Suppose UDP holds on $G_n$. We want to show that Conditions (i) and (ii) both hold. To show Condition (i) holds, notice that for all $1\leq j\leq n-1$, $G_j$ is a subgraph of $G_n$ and for every $u,w \in V(G_j)$ we have $\mathcal{P}_{G_j}(u,w)=\mathcal{P}_{G_n}(u,w)$. Since UDP holds on $G_n$, by Theorem \ref{sub}, UDP must also hold on $G_j$. We will use Theorem \ref{paste} to show that Condition (ii) holds as well. Note that $G_n = G_{n-1} \cup T_n$, where $G_{n-1}$ and $T_n$ both satisfy UDP, $V(G_{n-1})\cap V(T_n) = \{v_n\}$, and $G_n$ is created by pasting $G_{n-1}$ and $T_n$ at the vertex $v_n$.  Since UDP holds on $G_n$, Theorem \ref{paste} tells us that for every $u\in V(G_{n-1})$ and $w\in V(T_n)$, $\bigcap_{P\in\mathcal{P}(u,w)}\alpha(P)=\left(\bigcap_{P\in\mathcal{P}(u,v_n)}\alpha(P)\right) + \left(\bigcap_{P\in\mathcal{P}(v_n,w)}\alpha(P)\right)$, and this is Condition (ii).
    
    Now suppose Conditions (i) and (ii) both hold.  Then $G_{n-1}$ and $T_n$ both satisfy UDP.  Also, $G_n=G_{n-1}\cup T_n$, where $V(G_{n-1})\cap V(T_n)=\{v_n\}$ and $G_n$ is created by pasting together the graphs $G_{n-1}$ and $T_n$ at the vertex $v_n$.  Because Condition (ii) holds, we can use Theorem \ref{paste} to conclude that UDP holds on $G_n$.
    
\end{proof}


\section{Graph Subdivisions and Multigraphs}
\label{extensions}

Consider a graph formed by pasting together at a single vertex two edge-labeled graphs each satisfying UDP. Conditions for when UDP is satisfied on such a pasted graph, like the condition given in Theorem \ref{paste}, were explored in \cite{UDP}.  Figures $7$ and $8$ in \cite{UDP} contain examples of such pasted graphs that do not satisfy UDP.  A natural next step is to consider pasting graphs together along an edge rather than at a vertex.  Specifically, given two graphs $G_1$ and $G_2$ satisfying UDP and pasting $G_1$ and $G_2$ along an edge to create the graph $G$, must $G$ satisfy UDP?  If not, how can the edges be labeled such that $G$ does not satisfy UDP? 

One of the simplest examples of such an edge-pasted graph is the diamond graph because a diamond graph can be viewed as two 3-cycles pasted along a common edge. An example of an edge-labeled diamond graph is given in Figure \ref{DG}(c). To construct an edge-labeling on the diamond graph such that UDP is not satisfied, we  define the subdivision of an edge-labeled graph and show that if an edge-labeled graph does not satisfy UDP, then no subdivision of the graph can satisfy UDP. Using this idea of subdivision in conjunction with known results on the set of splines on multigraphs, we present an edge-labeling on the diamond graph such that UDP does not hold on the edge-labeled diamond.

We begin with two definitions from \cite{West text}.
\begin{definition}\label{def:West subdivision of edge}\cite[Definition 4.2.5]{West text}
In a graph $G$, \textit{subdivision} of an edge $uv$ is the operation of replacing $uv$ with a path $u,w,v$ through a new vertex $w$.
\end{definition}

\begin{definition}\label{def:West subdivision of graph}\cite[Definition 5.2.19]{West text}
A \textit{$G$-subdivision} or \textit{subdivision of $G$} is a graph obtained from a graph $G$ by successive edge subdivisions.  Equivalently, it is a graph obtained from $G$ by replacing edges with pairwise internally disjoint paths.
\end{definition}

Figure \ref{fig:basic subdivision example} contains an example of a graph $G$ on four vertices and a subdivision of $G$ obtained from three edge subdivisions.  Notice that the subdivision of $G$ contains three additional vertices.

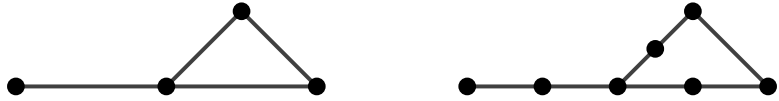
\begin{figure}[h! bt]
    \centering
    \begin{tikzpicture}
    \Vertex[size=0.2, color=black]{1}
    \Vertex[size=0.2, color=black, x=2]{2}
    \Vertex[size=0.2, color=black, x=4]{3}
    \Vertex[size=0.2, color=black, x=3, y=1]{4}
    \Edge(1)(2)
    \Edge(2)(3)
    \Edge(2)(4)
    \Edge(3)(4)
    \Vertex[size=0.2, color=black, x=6]{5}
    \Vertex[size=0.2, color=black, x=8]{6}
    \Vertex[size=0.2, color=black, x=10]{7}
    \Vertex[size=0.2, color=black, x=9, y=1]{8}
    \Vertex[size=0.2, color=black, x=9, fontcolor=black, position=above]{9}
    \Vertex[size=0.2, color=black, x=7, fontcolor=black, position=above]{10}
    \Vertex[size=0.2, color=black, x=8.5, y=0.5, fontcolor=black, position={above left}]{11}
    \Edge(5)(10)
    \Edge(10)(6)
    \Edge(6)(9)
    \Edge(9)(7)
    \Edge(6)(11)
    \Edge(11)(8)
    \Edge(7)(8)
    \end{tikzpicture}
    \caption{A graph $G$ and a subdivision of $G$ resulting from $3$ edges subdivisions}
    \label{fig:basic subdivision example}
\end{figure}

In this next definition, we extend the notion of subdivision of a graph to subdivision of an edge-labeled graph.

\begin{definition}\label{def:subdivision of edge-labeled graph}
A \textit{subdivision of an edge-labeled graph} $(G, \alpha)$ over a ring $R$ is an edge-labeled graph $(G', \alpha')$, where $G'$ is a subdivision of $G$ and $\alpha':E(G')\rightarrow{R}$ is an edge-labeling on $G'$ such that for each edge $e'\in E(G')$, 
\[\alpha'(e')=\begin{cases}
\alpha(e') & \text{if } e'\in E(G) \\
\alpha(e) & \text{if } e'\not\in E(G) \text{ and $e$ is the edge in $G$ that was subdivided to create the edge $e'$ in $G'$}.\\
\end{cases} 
\]  
\end{definition}

\begin{ex}

Consider the edge-labeled graphs given in Figure \ref{fig:extended_graph}, where $I_1, I_2, I_3,I_4$ are ideals of an arbitrary ring. The edge-labeled graph $G'$ on the right is a subdivision of the graph $G$ on the left. The edge-labeling on $G'$ is as described in Definition \ref{def:subdivision of edge-labeled graph}. Specifically, the only edge in $G'$ that was also an edge in $G$ is the edge with label $I_4$, and it has the same label in $G'$ as it does in $G$.  Every other edge in $G'$ was created by subdividing an edge in $G$ and has the same label as the edge that was subdivided to create the edge in $G'$.   
\end{ex}

\begin{figure}[h! bt]
    \centering
    \begin{tikzpicture}
    \Vertex[size=0.2, color=black]{1}
    \Vertex[size=0.2, color=black, x=2]{2}
    \Vertex[size=0.2, color=black, x=4]{3}
    \Vertex[size=0.2, color=black, x=3, y=1]{4}
    \Edge[label=$I_1$, position=below](1)(2)
    \Edge[label=$I_2$, position=below](2)(3)
    \Edge[label=$I_3$, position={above left}](2)(4)
    \Edge[label=$I_4$, position={above right}](3)(4)
    \Vertex[size=0.2, color=black, x=6]{5}
    \Vertex[size=0.2, color=black, x=8]{6}
    \Vertex[size=0.2, color=black, x=10]{7}
    \Vertex[size=0.2, color=black, x=9, y=1]{8}
    \Vertex[size=0.2, color=black, x=9, fontcolor=black, position=above]{9}
    \Vertex[size=0.2, color=black, x=7, fontcolor=black, position=above]{10}
    \Vertex[size=0.2, color=black, x=8.5, y=0.5, fontcolor=black, position={above left}]{11}
    \Edge[label=$I_1$, position=below](5)(10)
    \Edge[label=$I_1$, position=below](10)(6)
    \Edge[label=$I_2$, position=below](6)(9)
    \Edge[label=$I_2$, position=below](9)(7)
    \Edge[label=$I_3$, position={above left}](6)(11)
    \Edge[label=$I_3$, position={above left}](11)(8)
    \Edge[label=$I_4$, position={above right}](7)(8)
    \end{tikzpicture}
    \caption{Edge-labeled graph and an expansion}
    \label{fig:extended_graph}
\end{figure}
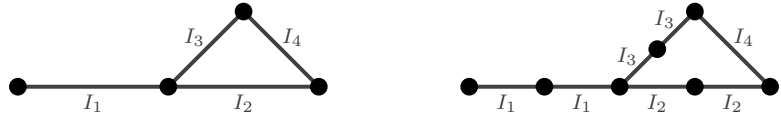

\begin{lem}\label{lem:paths with matching sums}
Let $(G,\alpha)$ be an edge-labeled graph and $(G',\alpha')$ be a subdivision of $(G,\alpha)$.  For any $uv$-path $P$ in $G$, there exists a $uv$-path $P'$ in $G'$ such that $\alpha'(P')=\alpha(P)$.  Also, for any $uv$-path $P'$ in $G'$, there exists a $uv$-path $P$ in $G$ such that $\alpha(P)=\alpha'(P')$.
\end{lem}

\begin{proof}
Let $(G,\alpha)$ be an edge-labeled graph and $(G',\alpha')$ be a subdivision of $(G,\alpha)$.  Let $P$ be an arbitrary $uv$-path in $G$ and $P'$ be the $uv$-path in $G'$ such that every edge in $P'$ is either an edge in $P$ or was created by subdividing an edge in $P$.  Then every edge label in $P'$ is an edge label in $P$ and vice versa.  Thus $\alpha'(P')=\alpha(P)$.

It remains to show that every $uv$-path $P'$ in $G'$ has a corresponding $uv$-path $P$ in $G$ such that $\alpha(P)=\alpha'(P')$.  Let $P'$ be an arbitrary path in $G'$ from $u$ to $v$ and $P$ be the $uv$-path in $G$ such that every edge in $P$ is either an edge in $P'$ or was subdivided to create an edge in $P'$.  Then every edge label in $P$ is an edge label in $P'$ and vice versa.  Thus $\alpha(P)=\alpha'(P')$.

\end{proof}

\begin{lem}\label{intersection_paths}

Let $(G,\alpha)$ be an edge-labeled graph and $(G',\alpha')$ be a subdivision of $(G,\alpha)$.  For any $u,v\in V(G)$, it must be that \[\bigcap\limits_{P\in\mathcal{P}_G(u,v)}\alpha(P)=\bigcap\limits_{P'\in\mathcal{P}_{G'}(u,v)}\alpha'(P').\]

\end{lem}

\begin{proof} 
Let $u,v\in V(G)$ and $a$ be an element of $\bigcap\limits_{P\in\mathcal{P}_G(u,v)}\alpha(P)$.  Let $Q'$ be an arbitrary path from $u$ to $v$ in the subdivision $G'$.  We want to show that $a$ must belong to $\alpha'(Q')$.  By Lemma \ref{lem:paths with matching sums}, we know there exists a path $Q$ from $u$ to $v$ in $G$ such that $\alpha(Q)=\alpha'(Q')$.  Since $a$ belongs to $\alpha(P)$ for every path $P$ from $u$ to $v$ in $G$, we know $a\in \alpha(Q)$.  Thus $a\in\alpha'(Q')$.

Now let $b$ be an element of $\bigcap\limits_{P'\in\mathcal{P}_{G'}(u,v)}\alpha'(P')$.  Let $S$ be an arbitrary path from $u$ to $v$ in $G$.  We want to show that $b$ must belong to $\alpha(S)$.  By Lemma \ref{lem:paths with matching sums}, we know there exists a path $S'$ from $u$ to $v$ in $G'$ such that $\alpha'(S')=\alpha(S)$.  Since $b$ belongs to $\alpha'(P')$ for every path $P'$ from $u$ to $v$ in $G'$, we know $b\in \alpha'(S')$.  Thus $b\in\alpha(S)$.

\end{proof}

In \cite{UDP}, Alt{\i}nok et al. only considered UDP as a property of edge-labeled simple graphs, and until now we have restricted our attention to simple graphs as well.  In Theorem \ref{thm:UDP on diamond graph}, we will show that UDP does not have to hold on a diamond graph, and in the proof of that theorem, we will consider an edge-labeled multigraph.  Note that the definition of UDP as it is written in Definition \ref{def:UDP} can be extended to define UDP on an edge-labeled multigraph.  The definitions for subdivision of an edge, of a graph, and of an edge-labeled graph given in Definitions \ref{def:West subdivision of edge}, \ref{def:West subdivision of graph}, and \ref{def:subdivision of edge-labeled graph} also extend as written for edge-labeled multigraphs.  Lastly, note that Lemma \ref{lem:paths with matching sums} and Lemma \ref{intersection_paths} and their proofs hold as written for edge-labeled multigraphs.  Our next result, Theorem \ref{extension} holds whether the graph $(G,\alpha)$ and its subdivision $(G',\alpha')$ are simple graphs or multigraphs.  We prove the result for multigraphs, and the simple graphs are just a special case.  We will apply this result to an edge-labeled multigraph in our proof of Theorem \ref{thm:UDP on diamond graph}.

\begin{thm}\label{extension}
If UDP does not hold on an edge-labeled graph, then UDP does not hold on any subdivision of the edge-labeled graph.
\end{thm}

\begin{proof}

    Let $(G, \alpha)$ be a graph that does not satisfy UDP and $(G', \alpha')$ be a subdivision of $(G, \alpha)$. Because ($G, \alpha$) does not satisfy UDP, there exist $u, v\in{V(G)}$ and some $x\in\bigcap_{P\in\mathcal{P}_G(u,v)}\alpha(P)$ such that there exists no spline $\rho$ on $(G, \alpha)$ for which $\rho(u)-\rho(v)=x$. For the sake of contradiction, suppose that UDP does hold on $(G', \alpha')$. By Lemma \ref{intersection_paths}, $\bigcap_{P\in\mathcal{P}_G(u,v)}\alpha(P)=\bigcap_{P'\in\mathcal{P}_{G'}(u,v)}\alpha'(P')$, so we see that $x\in\bigcap_{P'\in\mathcal{P}_{G'(u,v)}}\alpha'(P')$. Since UDP holds on $(G', \alpha')$, there exists a spline $\rho'$ on ($G', \alpha'$) such that $\rho'(u)-\rho'(v)=x$. 
    
    Consider arbitrary adjacent vertices $z_1, z_2\in{V(G)}$.  Let $e$ be an edge between $z_1$ and $z_2$ in $G$ and let $I=\alpha(e)$.  If $e$ is not one of the edges subdivided in creating $G'$, then $z_1$ and $z_2$ are adjacent in $G'$, $e$ is an edge in $G'$, and $\alpha'(e)=I$.  Since $\rho'$ is a spline on $(G',\alpha')$, we know $\rho'(z_1)-\rho'(z_2)\in I$.

    Now suppose $e$ is one of the edges subdivided in creating $G'$.  Let $z_1', z_2', ..., z_k'$ be the internal vertices on the path from $z_1$ to $z_2$ in $G'$ that was created in subdividing $e$. Then $\langle{z_1, z_1', z_2', ..., z_k', z_2}\rangle$ is a path from $z_1$ to $z_2$ in $G'$, and all edges in this path are labeled $I$. Since $\rho'$ is a spline on ($G', \alpha'$), $\rho'(z_1)-\rho'(z_1')\in{I}$, $\rho'(z_i')-\rho'(z_{i+1}')\in{I}$, for any $1\leq{i}\leq{k-1}$, and $\rho'(z_k')-\rho'(z_2)\in{I}$. Thus \[\rho'(z_1)-\rho'(z_2)=\rho'(z_1)-\rho'(z_1')+\rho'(z_1')-\rho'(z_2')+\rho'(z_2')-\cdot\cdot\cdot-\rho'(z_k')+\rho'(z_k')-\rho'(z_2)\in{I}.\] 
    
    In both cases, we have $\rho'(z_1)-\rho'(z_2)\in I=\alpha(e)$.  Because this holds for any pair of adjacent vertices $z_1, z_2$ in $G$, defining $\rho$ on ($G, \alpha$) such that $\rho(w)=\rho'(w)$ for all $w\in{V(G)}$ creates a spline on ($G, \alpha$). Additionally, for this spline $\rho$, we have $\rho(u)-\rho(v)=\rho'(u)-\rho'(v)=x$, a contradiction. Thus, ($G', \alpha'$) cannot satisfy UDP.
    
\end{proof}

The following theorem, proven by Anders, Crans, Foster-Greenwood, Mellor, and Tymoczko in \cite{Anders_2020}, gives a process for reducing an edge-labeled multigraph to an edge-labeled simple graph and states that this process does not change the set of splines.  We will also use this theorem in our proof of Theorem \ref{thm:UDP on diamond graph}.

\begin{thm}\label{multigraph}\cite[Lemma 2.2]{Anders_2020}
Suppose $(G, \alpha)$ is a multigraph. Let $u, v\in{V(G)}$ be vertices such that there exist two edges $e_1$ and $e_2$ between $u$ and $v$. Let $(H, \alpha')$ be the graph obtained from $(G, \alpha)$ by replacing $e_1$ and $e_2$ with a single edge $e$ and defining the edge labeling $\alpha'$ by $\alpha'(e)=\alpha(e_1)\cap\alpha(e_2)$ and for any other edge $\Bar{e}$ in $H$, $\alpha'(\Bar{e})=\alpha(\Bar{e})$. Then the set of all splines on $(G, \alpha)$ is equal to the set of all splines on $(H, \alpha')$.
    
\end{thm}

In the case where $(G,\alpha)$ is an edge-labeled multigraph, we will refer to the corresponding edge-labeled simple graph $(H, \alpha')$ from Theorem \ref{multigraph} as the \textit{reduction} of $(G,\alpha)$ to a simple graph.

Using Theorems \ref{extension} and \ref{multigraph}, we now exhibit an example of an edge-labeled diamond graph on which UDP is not satisfied.


\begin{thm}\label{thm:UDP on diamond graph}

Over any ring $R$ with ideals $I,J,K$ for which $(I + K) \cap (J + K) \neq (I \cap J)+K$, there exists an edge-labeling on the diamond graph such that UDP does not hold. 

\end{thm}

\begin{proof}
Let R be a ring with ideals $I, J, K$, for which $(I+K)\cap(J+K)\neq{(I\cap{J})+K}$. Note that $(I \cap J)+K \subseteq (I+K) \cap (J + K)$, so it must be that $(I + K) \cap (J + K) \nsubseteq  (I \cap J)+K$. Let $x\in (I+K)\cap(J+K)$ such that $x\not\in(I\cap J)+K$.  Consider the three edge-labeled graphs MDG, NMDG, and DG shown in Figure \ref{DG}. Assume for the sake of contradiction that UDP holds on the graph MDG in Figure \ref{DG}(a).  In this multigraph, there are three distinct paths from $u$ to $v$, and 
\[
\bigcap_{P\in\mathcal{P}_{MDG}(u,v)}\alpha(P)=(I+K)\cap(I+K)\cap(J+K)=(I+K)\cap(J+K).
\]
Since UDP holds on MDG and $x\in (I+K)\cap(J+K)$, we know there exists a spline $\rho$ on MDG such that $\rho(u)-\rho(v)=x$.  Note that the graph NMDG in Figure \ref{DG}(b) is the reduction of MDG to a simple graph.  By Theorem \ref{multigraph}, $\rho$ must also be a spline on NMDG.  Now we can apply Theorem \ref{containment} to the spline $\rho$ on NMDG to conclude that
\[
x=\rho(u)-\rho(v)\in (I+K)\cap\left((I\cap J)+K)\right)=(I\cap J)+K,
\]
which is a contradiction to our choice of $x$.  Thus UDP does not hold on the graph MDG.  Because the graph DG in Figure \ref{DG}(c) is a subdivision of the edge-labeled multigraph NMDG, we see from Theorem \ref{extension} that UDP does not hold on DG.

\end{proof}

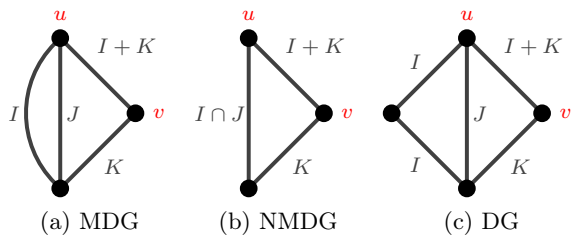
\begin{figure}[h! bt]
    \centering
    \subfloat[MDG]{
    \begin{tikzpicture}
    \Vertex [size=0.2, x=1, y=1, color=black, position=above, label=$u$, fontcolor=red] {1};
    \Vertex [size=0.2, x=2, color=black, position=right, label=$v$, fontcolor=red] {2};
    \Vertex [size=0.2, x=1, y=-1, color=black] {3};
    \Edge [bend=-45, label=$I$, position=left](1)(3)
    \Edge [label=$J$, position=right](1)(3)
    \Edge [label=$I+K$, position={above right=1mm}](1)(2)
    \Edge [label=$K$, position={below right=1mm}](3)(2)
    \end{tikzpicture}
    }
    \subfloat[NMDG]{
    \begin{tikzpicture}
    \Vertex [size=0.2, x=1, y=1, color=black, position=above, label=$u$, fontcolor=red] {1};
    \Vertex [size=0.2, x=2, color=black, position=right, label=$v$, fontcolor=red] {2};
    \Vertex [size=0.2, x=1, y=-1, color=black] {3};
    \Edge [label=$I\cap{J}$, position=left](1)(3)
    \Edge [label=$I+K$, position={above right=1mm}](1)(2)
    \Edge [label=$K$, position={below right=1mm}](3)(2)
    \end{tikzpicture} 
    }
    \subfloat[DG]{
    \begin{tikzpicture}
    \Vertex [size=0.2, x=1, y=1, color=black, position=above, label=$u$, fontcolor=red] {1};
    \Vertex [size=0.2, x=2, color=black, position=right, label=$v$, fontcolor=red] {2};
    \Vertex [size=0.2, x=1, y=-1, color=black] {3};
    \Vertex [size=0.2, x=0, color=black]{4}
    \Edge [label=$J$, position=right](1)(3)
    \Edge [label=$I+K$, position={above right=1mm}](1)(2)
    \Edge [label=$K$, position={below right=1mm}](3)(2)
    \Edge [label=$I$, position={above left=1mm}](4)(1)
    \Edge [label=$I$, position={below left=1mm}](4)(3)
    \end{tikzpicture} 
    }
    \caption{Constructing an edge-labeling on the diamond graph (DG) such that UDP does not hold}
    \label{DG}
\end{figure}


\section{The pairwise edge-disjoint path property and UDP}
\label{ped}

Alt{\i}nok et al. showed that UDP must hold on any edge-labeled cycle or tree over any ring $R$.  In Section \ref{unicyclic}, we showed that UDP does not hold on the edge-labeled unicyclic graph in Figure \ref{unicyclenoUDP}, and then we determined conditions under which UDP must hold on an edge-labeled unicyclic graph.  In Section \ref{extensions}, we shifted from considering graphs pasted together at a single vertex to graphs pasted together along a single edge.  Specifically, we considered the diamond graph, which can be seen as the result of pasting together two $3$-cycles along an edge.  In Theorem \ref{thm:UDP on diamond graph}, we saw an example of an edge-labeled diamond graph that does not satisfy UDP.  After considering these families of graphs, we began looking for a structural property of a graph that would indicate whether the graph must satisfy UDP for every possible edge-labeling over any ring $R$. This investigation led us to a condition we call the  pairwise edge-disjoint path property (PEDPP). We begin this section by defining the PEDPP,  then we present a list of conditions equivalent to the PEDPP, and we end this section by proving that a graph having the PEDPP must satisfy UDP for any edge-labeling over any ring.  In Section \ref{sec: subgraphs}, we will upgrade this result to the biconditional statement that given a graph $G$, UDP holds on $(G,\alpha)$ for any edge-labeling $\alpha$ over any ring $R$ if and only if $G$ satisfies the pairwise edge-disjoint path property.
\begin{definition}
   A connected graph $G$ has the \textit{pairwise edge-disjoint path property}, abbreviated PEDPP, if for any pair of distinct vertices $u$ and $v$ in $G$ and for any pair of distinct $u,v$-paths $P_1$ and $P_2$ in $G$, we have $E(P_1) \cap E(P_2) = \emptyset$.
\end{definition}
\begin{ex}
     Consider the 5-cycle on the right in Figure \ref{fig:PEDPP}.  Between any pair of distinct vertices, there are exactly two paths, and these paths do not share any edges. Hence the 5-cycle has the PEDPP. Now consider the graph on the left in Figure \ref{fig:PEDPP}. Here the $a,b$-paths $\langle a,e,b \rangle$ and $\langle a,e,d,b \rangle$ share the edge $ae$, so the graph does not have the PEDPP.
\end{ex}
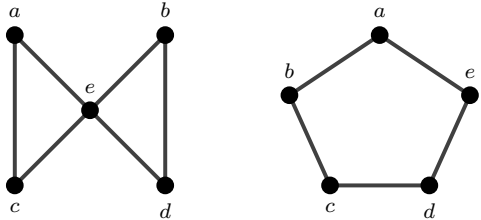
\begin{figure}[h! bt]
    \centering
    \begin{tikzpicture}
        \Vertex[size = 0.2, color = black, label = $e$, position = above]{1}
        \Vertex[size = 0.2, color = black, x = -1, y = 1, label = $a$, position = above]{2}
        \Vertex[size = 0.2, color = black, x = 1, y = 1, label = $b$, position = above]{3}
        \Vertex[size = 0.2, color = black, x = -1, y = -1, label = $c$, position = below]{4}
        \Vertex[size = 0.2, color = black, x = 1, y = -1, label = $d$, position = below]{5}

        \Edge(1)(2)
        \Edge(1)(3)
        \Edge(2)(4)
        \Edge(3)(5)
        \Edge(4)(1)
        \Edge(5)(1)
        
    \end{tikzpicture}
    \hspace{1cm}
    \begin{tikzpicture}
        \Vertex[size = 0.2, color = black, label = $a$, position = above]{1}
        \Vertex[size = 0.2, color = black, x = -1.2, y = -0.8, label = $b$, position = above]{2}
        \Vertex[size = 0.2, color = black, x = -0.66, y = -2, label = $c$, position = below]{3}
        \Vertex[size = 0.2, color = black, x = 0.66, y = -2, label = $d$, position = below]{4}
        \Vertex[size = 0.2, color = black, x = 1.2, y = -0.8, label = $e$, position = above]{5}

        \Edge(1)(2)
        \Edge(2)(3)
        \Edge(3)(4)
        \Edge(4)(5)
        \Edge(5)(1)
        
    \end{tikzpicture}
    \caption{The graph on the right satisfies the PEDPP, while the graph on the left does not.}
    \label{fig:PEDPP}
\end{figure}

We now introduce a basic graph theory definition from \cite{West text}.  This definition will be crucial in developing our description of a choke point in a graph, and then in Theorem \ref{thm:ChokePoint_PEDPP} we will make a connection between whether a graph has a choke point and whether the graph satisfies the PEDPP.  
\begin{definition}\label{def:West star}\cite[Example 2.1.2]{West text}
A \textit{star} is a tree consisting of one vertex adjacent to all the others.  A star $S$ with $|V(S)|=k+1$ is denoted by $S_k$, and we refer to the one vertex that is adjacent to all of the $k$ other vertices as the \textit{internal vertex} of the star.
\end{definition}

Figure \ref{fig:StarExample} contains an illustration of the star $S_3$.
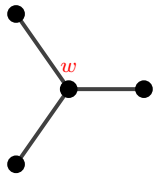
\begin{figure}[h! bt]
    \centering
    \begin{tikzpicture}
        \Vertex[size = 0.2, x = 2.3, y = 1, color = black]{4};
        \Vertex[size = 0.2, x = 2.3, y = -1, color = black]{5};
        \Vertex[size = 0.2, x = 3, color = black, label = $w$, fontcolor=red, position = above]{6};
        \Vertex[size = 0.2, x = 4, color = black]{7};
        \Edge(4)(6);
        \Edge(5)(6);
        \Edge(6)(7);
    \end{tikzpicture}
    \caption{A star $S_3$ with $4$ vertices and $w$ is the internal vertex of the star}
    \label{fig:StarExample}
\end{figure}

\begin{definition}
    Let $G$ be a graph and let $u,v \in V(G)$. Consider the subgraph $F = \bigcup \limits_{P \in \mathcal{P}_{(u,v)}} P$. If there exists a star $S_k \subseteq F$ with $k \geq 3$, then the internal vertex of $S_k$ is called a \textit{choke point} of $u$ and $v$.
\end{definition}
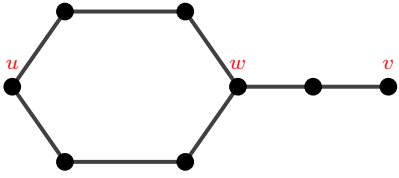
\begin{figure}[h! bt]
    \centering
    \begin{tikzpicture}
        \Vertex[size = 0.2, color = black, label = $u$, fontcolor=red, position = above]{1};
        \Vertex[size = 0.2, x = 0.7, y = 1, color = black]{2};
        \Vertex[size = 0.2, x = 0.7, y = -1, color = black]{3};
        \Vertex[size = 0.2, x = 2.3, y = 1, color = black]{4};
        \Vertex[size = 0.2, x = 2.3, y = -1, color = black]{5};
        \Vertex[size = 0.2, x = 3, color = black, label = $w$, fontcolor=red, position = above]{6};
        \Vertex[size = 0.2, x = 4, color = black]{7};
        \Vertex[size = 0.2, x = 5, color = black, label = $v$, fontcolor=red, position = above]{8};
        \Edge(1)(2);
        \Edge(1)(3);
        \Edge(2)(4);
        \Edge(3)(5);
        \Edge(4)(6);
        \Edge(5)(6);
        \Edge(6)(7);
        \Edge(7)(8);
    \end{tikzpicture}
    \caption{In this graph, the vertex $w$ is a choke point of $u$ and $v$.}
    \label{fig:chokePointExample}
\end{figure}
\begin{ex}
    Consider the graph in Figure \ref{fig:chokePointExample}.  Here the subgraph $F = \bigcup \limits_{P \in \mathcal{P}_{(u,v)}} P$ is the entire graph, and the vertex $w$ is a choke point of the vertices $u$ and $v$. In this graph, the two paths from $u$ to $v$ both pass through the choke point $w$ and share two edges, those on the path from $w$ to $v$.  Consequently, this graph does not satisfy the pairwise edge-disjoint path property.
\end{ex}
We now prove that $G$ having no choke points is equivalent to $G$ satisfying the PEDPP and also to $G$ being either a tree or cycle.
\begin{thm} \label{thm:ChokePoint_PEDPP}
    Let $G$ be a connected graph. The following are equivalent: 
    \begin{enumerate}
        \item $G$ has no choke points.
    
        \item $G$ satisfies the pairwise edge-disjoint path property.
    
        \item $G$ is a tree or a cycle.
    \end{enumerate}
\end{thm}
\begin{proof}
    We begin with a contrapositive proof that Condition 1 implies Condition 2.  Assume that $G$ does not satisfy the pairwise edge-disjoint path property. Then there exist distinct vertices $u$ and $v$ in $G$ and distinct paths $P_1, P_2 \in \mathcal{P}_G(u,v)$ such that $E(P_1) \cap E(P_2) \neq \emptyset$. 
    Consider $e \in E(P_1) \cap E(P_2)$ such that $e$ is adjacent to an edge $a$ strictly in $P_1$. Such an edge $e$ must exist since $P_1$ and $P_2$ are distinct paths from $u$ to $v$. Note that the edges $e$ and $a$ must both be incident to a common vertex $w$. Now, for the sake of contradiction, we assume that there does not exist an edge $b$ strictly in $P_2$ that is incident to $w$. Then $w$ must be an endpoint of the path $P_2$, so $w = u$ or $w = v$. The path $P_1$, however, does not end at $w$, so $w \not = u$ and $w \not = v$, and we have our contradiction. Thus there does exist an edge $b$ that is strictly in $P_2$ and is incident to $w$. Since the edges $a,e,b$ are all incident to $w$ and all belong to $\bigcup_{P \in \mathcal{P}(u,v)} P$, we see that $w$ is a choke point of $u$ and $v$.

    Next we give a contrapositive proof that Condition 2 implies Condition 1. Assume that G has a choke point $w$. Then there exist distinct vertices $u,v$ in $G$ such that $w$ is the internal vertex of a star $S_k$ with $k\geq 3$ and $S_k \subseteq\bigcup \limits_{P \in \mathcal{P}_{(u,v)}} P$.  Consequently, $w$ has at least 3 distinct incident edges in the set of all paths from u to v. It follows that $w$ has at least one incident edge, say $e_1$, on a path from $u$ to $w$, that $w$ has at least one incident edge, say $e_2$, on a path from $w$ to $v$, and that $w$ has a third distinct incident edge, say $e_3$, which must be either on a path from $u$ to $w$ or on a path from $w$ to $v$.  Without loss of generality, suppose $e_3$ is on a path from $w$ to $v$. Now define $P_1$ to be the path from $u$ to $v$ such that $P_1$ includes $e_1$ and $e_2$ but not $e_3$ and define $P_2$ to be the path from $u$ to $v$ such that $P_2$ includes $e_1$ and $e_3$ but not $e_2$. Then, $P_1$ and $P_2$ are distinct paths from $u$ to $v$ which share an edge. Thus $G$ does not satisfy the PEDPP. 
    
    Now we will prove that Condition 2 holds if and only if Condition 3 holds. To prove Condition 3 implies Condition 2, we consider two cases. First suppose $G$ is a tree. For any $u,v \in V(G)$, there is exactly one path from $u$ to $v$, so $G$ satisfies the PEDPP vacuously. Now consider the case where $G$ is a cycle. Let $u$ and $v$ be distinct vertices in $G$. There are exactly two distinct paths $P_1$ and $P_2$ from $u$ to $v$ in $G$. Because $G$ is a cycle, $E(P_1) \cap E(P_2) = \emptyset$, and $G$ satisfies the PEDPP.
    
    Lastly, we must show that Condition 2 implies Condition 3. Let $G$ be a connected graph satisfying the pairwise edge-disjoint path property. For the sake of contradiction, assume that $G$ is not a tree or cycle. Because $G$ is not a tree, $G$ must contain a cycle.  Because $G$ is also not a cycle, $G$ must contain a vertex $u$ that is not in the cycle but is connected to the cycle by an edge $e$. Let $v$ be a vertex in the cycle not incident to $e$. We can find two distint $u,v$-paths that pass into the cycle through $e$, which is a contradiction to G satisfying the PEDPP. Hence $G$ must be a tree or cycle. 

\end{proof}

We combine Theorem \ref{thm:ChokePoint_PEDPP} with results from Alt{\i}nok et al., stated here as Theorems \ref{UPDTreeThm} and \ref{UDPCycleThm}, to obtain the following corollary.

\begin{cor}
    Let $G$ be a graph satisfying the pairwise edge-disjoint path property. Then UDP must hold on $(G,\alpha)$ for any edge-labeling $\alpha$ over any ring $R$.
\end{cor}
\begin{proof}
    Let $(G, \alpha)$ be an edge-labeled graph over a ring R such that $G$ has the pairwise edge-disjoint path property. By Theorem \ref{thm:ChokePoint_PEDPP}, $G$ is a tree or cycle, and we know from Theorems \ref{UPDTreeThm} and \ref{UDPCycleThm} that UDP holds on any edge-labeled tree or cycle over any ring $R$.

\end{proof}


\section{Subgraphs and UDP}
\label{sec: subgraphs}

In this section, we explore the relationship between UDP and subgraphs. If we have an edge-labeled graph $(G,\alpha)$ and an induced edge-labeled subgraph $(H,\alpha')$, when considering whether or not $(G,\alpha)$ satisfies UDP and whether or not $(H,\alpha')$ satisfies UDP, there are four possible combinations.  We show that all four combinations are realizable.  Then, in our next result, we show that for any ring $R$, any graph $G$, and any subgraph $G'$, if UDP holds on $(G,\alpha)$ for every edge-labeling $\alpha$ over $R$, then UDP must hold on $(G',\beta)$ for any edge-labeling $\beta$ over $R$.  Using this, we are then able to give a complete classification of families of graphs for which UDP must be satisfied for any graph with any edge-labeling over any ring $R$.  Specifically, trees and cycles are the only graphs for which UDP must be satisfied for any edge-labeling over any ring.  We conclude by showing that when $R$ is an integral domain, if there exists a graph $G$ that is neither a tree nor cycle such that every edge-labeling of $G$ over $R$ satisfies UDP, then $R$ is a Prüfer domain.  This allows us to check only edge-labelings over a certain unicyclic graph rather than having to check edge-labelings over every possible graph.

We begin this section by exploring the relationship between UDP on an edge-labeled graph and UDP on a corresponding induced edge-labeled subgraph.

\begin{ex}
Let $R$ be an integral domain that is not a Prüfer domain.  Then there exist non-zero ideals $I,J,K$ of $R$ such that $I+(J\cap K)\neq (I+J)\cap(I+K)$.  Consider the edge-labeled graphs over this integral domain $R$ in Figures \ref{GDHDN}, \ref{GDHD}, \ref{GDNHD}, and \ref{GDNHDN}. In each figure, the graph on the left is an edge-labeled graph $(G, \alpha)$, while the graph on the right is a subgraph $H$ of $G$ with the induced edge-labeling $\alpha'$. In these four figures, we demonstrate the possible combinations that can arise from UDP holding or not holding on an edge-labeled graph and then UDP holding or not holding on a subgraph with the induced edge-labeling. The examples in these figures illustrate that all four possible combinations can indeed occur.

First we consider the edge-labeled graph $(G,\alpha)$ on the left in Figure \ref{GDHDN}.  This graph satisfies the UDP.  This is not immediately obvious, but it can be verified by hand, checking that for each possible pair of distinct vertices $u$ and $v$ in $G$, and for every $x$ in $\bigcap\limits_{P\in\mathcal{P}_G(u,v)}\alpha(P)$, there exists a spline $\rho$ on $G$ such that $\rho(u)-\rho(v)=x$.

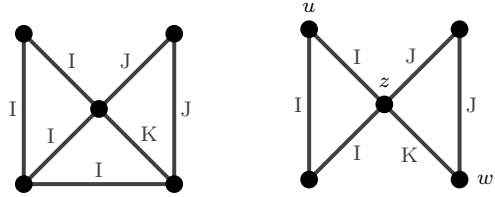
\begin{figure}[h! bt]
    \centering
    \centering
    \begin{tikzpicture}
    \Vertex [size=0.2, x=1, y=1, color=black] {2};
    \Vertex [size=0.2, x=1, y=-1, color=black] {3};
    \Vertex [size=0.2, x=2, color=black] {4};
    \Vertex [size=0.2, x=3, y=1, color=black] {5};
    \Vertex [size=0.2, x=3, y=-1, color=black] {6};
    \Edge [label=I, position=left](2)(3);
    \Edge [label=I, position={above right=0.5mm}](2)(4);
    \Edge [label=I, position={above left=0.5mm}](3)(4);
    \Edge [label=J, position={above left=0.5mm}](4)(5);
    \Edge [label=K, position={above right=0.5mm}](4)(6);
    \Edge [label=J, position=right](5)(6);
    \Edge [label=I, position={above=0.5mm}](3)(6);
    \end{tikzpicture} 
    \hspace{1cm}
    \begin{tikzpicture}
    \Vertex [size=0.2, x=1, y=1, color=black, label = $u$, fontcolor = black, position = above] {2};
    \Vertex [size=0.2, x=1, y=-1, color=black] {3};
    \Vertex [size=0.2, x=2, color=black, label = $z$, fontcolor = black, position = above] {4};
    \Vertex [size=0.2, x=3, y=1, color=black] {5};
    \Vertex [size=0.2, x=3, y=-1, color=black, label = $w$, fontcolor = black, position = right] {6};
    \Edge [label=I, position=left](2)(3);
    \Edge [label=I, position={above right=0.5mm}](2)(4);
    \Edge [label=I, position={below right=0.5mm}](3)(4);
    \Edge [label=J, position={above left=0.5mm}](4)(5);
    \Edge [label=K, position={below left=0.5mm}](4)(6);
    \Edge [label=J, position=right](5)(6);
    \end{tikzpicture} 
    \caption{UDP holds on $(G, \alpha)$ but does not hold on $(H, \alpha')$}
    \label{GDHDN}
\end{figure}
Now we consider the edge-labeled subgraph $(H,\alpha')$ on the right in Figure \ref{GDHDN}.  This subgraph can be viewed as the result of gluing together two $3$-cycles at the vertex $z$.  Since we know UDP is satisfied on each of the $3$-cycles, we can use Theorem \ref{paste} to determine whether UDP holds on $(H,\alpha')$.    Specifically , we will show that Equation \eqref{eq:key pasted graphs equation} does not hold for this choice of vertices $u$ and $w$ in $(H,\alpha')$.  First note that  
\[
\left( \bigcap \limits_{P \in \mathcal{P}_G(u,z)} \alpha(P) \right) + \left( \bigcap \limits_{P \in \mathcal{P}_G(z,w)} \alpha(P) \right)=I+(J\cap K).  
\]
Meanwhile, $\bigcap \limits_{P \in \mathcal{P}_G(u,w)} \alpha(P)=(I+J)\cap(I+K)$.  By our choice of $I,J,K$, we see that Equation \eqref{eq:key pasted graphs equation} does not hold for these vertices $u$ and $w$, so UDP does not hold on $(H,\alpha')$.

Next we turn our attention to the graphs in Figure \ref{GDHD}.  The graph $G$ on the left is a cycle, so UDP will be satisfied on $G$ with any edge-labeling $\alpha$ over any ring $R$, as stated in Theorem \ref{UDPCycleThm}.  Similarly, the subgraph $H$ in Figure 11 is a path, so we know from Theorem \ref{UPDTreeThm} that UDP will be satisfied on $H$ with any edge-labeling over any ring.

\begin{figure}[h! bt]
    \centering
    \begin{tikzpicture}
    \Vertex [size=0.2, color=black] {1};
    \Vertex [size=0.2, x=1, color=black] {2};
    \Vertex [size=0.2, y=1, color=black] {3};
    \Vertex [size=0.2, x=1, y=1, color=black] {4};
    \Edge [label=I, position={below=1mm}](1)(2);
    \Edge [label=J, position={left=1mm}](1)(3);
    \Edge [label=K, position={right=1mm}](2)(4);
    \Edge [label=L, position={above=1mm}](3)(4);
    \end{tikzpicture}
    \hspace{1cm}
    \begin{tikzpicture}
    \Vertex [size=0.2, color=black] {1};
    \Vertex [size=0.2, x=1, color=black] {2};
    \Vertex [size=0.2, y=1, color=black] {3};
    \Vertex [size=0.2, x=1, y=1, color=black] {4};
    \Edge [label=I, position={below=1mm}](1)(2);
    \Edge [label=J, position={left=1mm}](1)(3);
    \Edge [label=K, position={right=1mm}](2)(4);
    \end{tikzpicture}
    \caption{UDP holds on $(G, \alpha)$ and UDP holds on $(H, \alpha')$}
    \label{GDHD}
\end{figure}
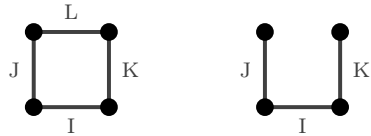
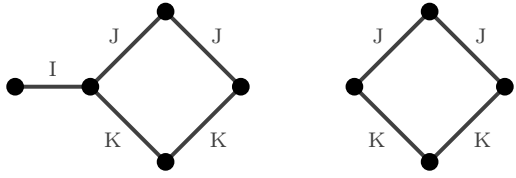
\begin{figure}[h!]
    \centering
    \centering
    \begin{tikzpicture}
    \Vertex [size=0.2, color=black] {1};
    \Vertex [size=0.2, x=1, y=1, color=black] {2};
    \Vertex [size=0.2, x=2, color=black] {3};
    \Vertex [size=0.2, x=1, y=-1, color=black] {4};
    \Vertex [size=0.2, x=-1, color=black] {5}
    \Edge [label=J, position={above left=1mm}](1)(2);
    \Edge [label=J, position={above right=1mm}](2)(3);
    \Edge [label=K, position={below left=1mm}](1)(4);
    \Edge [label=K, position={below right=1mm}](3)(4);
    \Edge [label=I, position={above=1mm}](1)(5);
    \end{tikzpicture}
    \hspace{1cm}
    \begin{tikzpicture}
    \Vertex [size=0.2, color=black] {1};
    \Vertex [size=0.2, x=1, y=1, color=black] {2};
    \Vertex [size=0.2, x=2, color=black] {3};
    \Vertex [size=0.2, x=1, y=-1, color=black] {4};
    \Edge [label=J, position={above left=1mm}](1)(2);
    \Edge [label=J, position={above right=1mm}](2)(3);
    \Edge [label=K, position={below left=1mm}](1)(4);
    \Edge [label=K, position={below right=1mm}](3)(4);
    \end{tikzpicture}
    \caption{UDP does not hold on $(G, \alpha)$ but does hold on $(H, \alpha')$}
    \label{GDNHD}
\end{figure}
We now consider the graphs in Figure \ref{GDNHD}.  The graph $(G,\alpha)$ on the left was discussed in Example \ref{ex:unicyclic graph not UDP}, where it was shown that this graph does not satisfy UDP.  Because the subgraph $H$ on the right is a cycle, it must satisfy UDP for any edge-labeling over any ring $R$.

Lastly, we consider the graphs in Figure \ref{GDNHDN}.  We can view the graph $(G,\alpha)$ as having been formed by pasting together at the vertex $z$ the path consisting of two edges and the $4$-cycle.  Since UDP must be satisfied on the path and on the $4$-cycle, we can once again use Theorem \ref{paste} to determine whether UDP holds on $(G,\alpha)$.    Specifically , we will show that Equation \eqref{eq:key pasted graphs equation} does not hold for this choice of vertices $u$ and $w$ in $(G,\alpha)$.  First note that  
\[
\left( \bigcap \limits_{P \in \mathcal{P}_G(u,z)} \alpha(P) \right) + \left( \bigcap \limits_{P \in \mathcal{P}_G(z,w)} \alpha(P) \right)=I+(J\cap K).  
\]
Meanwhile, $\bigcap \limits_{P \in \mathcal{P}_G(u,w)} \alpha(P)=(I+J)\cap(I+K)$.  By our choice of $I,J,K$, we see that Equation \eqref{eq:key pasted graphs equation} does not hold for these vertices $u$ and $w$, so UDP does not hold on $(G,\alpha)$.  Finally, we again note that the subgraph $(H, \alpha')$ on the right in Figure \ref{GDNHDN} was discussed in Example 2.2, where it was shown that this graph does not satisfy UDP.

\begin{figure}[h! bt]
    \centering
    \begin{tikzpicture}
    \Vertex [size=0.2, color=black, label=$z$, fontcolor= black, position = above] {1};
    \Vertex [size=0.2, x=1, y=1, color=black] {2};
\Vertex [size=0.2, x=2, color=black, label=$w$, fontcolor= black, position = above] {3};
    \Vertex [size=0.2, x=1, y=-1, color=black] {4};
    \Vertex [size=0.2, x=-1, color=black] {5}
    \Vertex [size=0.2, x=-2, color=black,label=$u$, fontcolor= black, position = above] {6}
    \Edge [label=J, position={above left=1mm}](1)(2);
    \Edge [label=J, position={above right=1mm}](2)(3);
    \Edge [label=K, position={below left=1mm}](1)(4);
    \Edge [label=K, position={below right=1mm}](3)(4);
    \Edge [label=I, position={above=1mm}](1)(5);
    \Edge [label=I, position={above=1mm}](6)(5);
    \end{tikzpicture}
    \hspace{1cm}
    \begin{tikzpicture}
    \Vertex [size=0.2, color=black] {1};
    \Vertex [size=0.2, x=1, y=1, color=black] {2};
    \Vertex [size=0.2, x=2, color=black] {3};
     \Vertex [size=0.2, x=2, color=black] {4};
    \Vertex [size=0.2, x=1, y=-1, color=black] {4};
    \Vertex [size=0.2, x=-1, color=black] {5}
    \Edge [label=J, position={above left=1mm}](1)(2);
    \Edge [label=J, position={above right=1mm}](2)(3);
    \Edge [label=K, position={below left=1mm}](1)(4);
    \Edge [label=K, position={below right=1mm}](3)(4);
    \Edge [label=I, position={above=1mm}](1)(5);
    \end{tikzpicture}
    \caption{UDP does not hold on $(G, \alpha)$ nor does UDP hold on $(H, \alpha')$}
    \label{GDNHDN}
\end{figure}
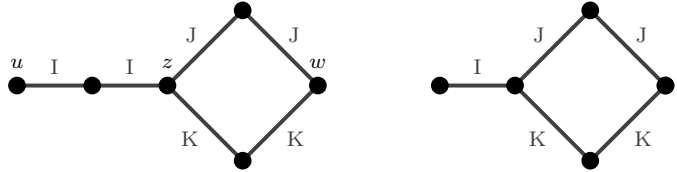
\end{ex}

From Figures \ref{GDHDN}, \ref{GDHD}, \ref{GDNHD}, and \ref{GDNHDN}, we see that if UDP holds on an edge-labeled graph, it may or may not hold on an induced edge-labeled subgraph, and we also see that if UDP does not hold on an edge-labeled graph, it may or may not hold on an induced edge-labeled subgraph.  This concludes our illustration that all four possible combinations are realizable.

In the next lemma, we consider a special case.  We are not just looking at a case where UDP holds on a particular edge-labeled graph over a certain ring $R$, as we did in Figures \ref{GDHDN} and \ref{GDHD}, but we are instead considering the case in which UDP holds on a graph $G$ for any edge-labeling $\alpha$ over any ring $R$.  This lemma will be helpful in classifying what types of graphs must satisfy UDP for any edge-labeling over any ring $R$. 

\begin{lem}\label{Every}
    Let $R$ be a ring and let $G$ be a graph with subgraph $G'$. If UDP holds on $(G,\alpha)$ for every edge-labeling $\alpha$ over $R$, then UDP holds on $(G',\beta)$ for every edge-labeling $\beta$ over $R$.
\end{lem}

\begin{proof}
    \noindent Let $\beta$ be an edge-labeling of $G'$, let $u,v \in V(G')$, and let \mbox{$x \in \bigcap_{P \in \mathcal{P}_{G'}(u,v)} \beta(P)$.} Ultimately, to show that UDP holds on $(G', \beta)$, we will construct a spline $\rho'$ on $G'$ such that $\rho'(u) - \rho'(v) = x$. Our next step is to  define an edge-labeling $\alpha$ on $G$. For any edge $e\in E(G)$, let $\alpha(e) = \beta(e)$ if $e \in E(G')$ and $\alpha(e) = \langle{1}\rangle$ if $e \notin E(G')$. Note that there are exactly two types of paths from $u$ to $v$ in $G$. Let $\mathcal{P}_{1,G}$ denote the set of all paths from $u \text{ to } v$ in $G$ such that the path is contained entirely in $G'$. Let $\mathcal{P}_{2,G}$ denote the set of all paths from $u \text{ to } v$ in $G$ such that the path contains at least one edge not in $G'$. Then we have \[\bigcap\limits_{P \in \mathcal{P}_G(u,v)} \alpha(P) = \left (\bigcap\limits_{P \in \mathcal{P}_{1,G}(u,v)} \alpha(P) \right ) \cap \left (\bigcap\limits_{P \in \mathcal{P}_{2,G}(u,v)} \alpha(P)\right ).\]
    However, since each path in $P_2$ contains an edge labeled $\langle{1}\rangle$, we have \[\bigcap\limits_{P \in \mathcal{P}_{2,G}(u,v)} \alpha(P) = \langle{1}\rangle,\] and hence \[\bigcap\limits_{P \in \mathcal{P}_G(u,v)} \alpha(P) = \left (\bigcap\limits_{P \in \mathcal{P}_{1,G}(u,v)} \alpha(P) \right ) \cap \langle{1}\rangle = \bigcap\limits_{P \in \mathcal{P}_{G'}(u,v)} \beta(P).\] Recall that $x \in \bigcap_{P \in \mathcal{P}_{G'}(u,v)} \beta(P)$. Thus $x \in \bigcap_{P \in \mathcal{P}_{G}(u,v)} \alpha(P)$. Since UDP holds on $G$ for every edge-labeling over $R$, we can choose a spline $\rho$ on $G$ with $\rho(u) - \rho(v) = x$. Define $\rho': V(G') \to R$ by $\rho'(w) = \rho(w)$ for all $w \in V(G')$. Then $\rho'$ is a spline on $G'$ with $\rho'(u) - \rho'(v) = x$. 
\end{proof}

We are now prepared to show that trees and cycles are the only graphs for which UDP must be satisfied for any edge-labeling $\alpha$ over any ring $R$.

\begin{thm}\label{thm:complete characterization}
    \noindent Given a graph $G$, UDP holds on $(G,\alpha)$ for any edge labeling $\alpha$ over any ring $R$ if and only if $G$ is a tree or a cycle.
\end{thm}

\begin{proof}
    \noindent Suppose that $G$ is a tree or a cycle. Then from Theorems \ref{UPDTreeThm} and \ref{UDPCycleThm}, we have that UDP holds on $(G,\alpha)$ for any edge labeling $\alpha$ over any ring $R$.
\\

    \begin{figure}[hbt!]
            \centering
            \begin{tikzpicture}
                \Vertex [size=.2, color=black, label=$z$, position=above, fontcolor=black] {1};
                \Vertex [size=.2, x=1, y=1, color=black] {2};
                \Vertex [size=.2, x=2, color=black, label=$w$, position=above, fontcolor=black] {3};
                \Vertex [size=.2, x=1, y=-1, color=black] {4};
                \Vertex [size=.2, x=-1, color=black, label=$u$, position=above, fontcolor=black] {5};
                \Edge [label=$J$, position={above left=1mm}](1)(2);
                \Edge [style=dashed, label=$J$, position={above right=1mm}](2)(3);
                \Edge [label=$K$, position={below left=1mm}](1)(4);
                \Edge [label=$I$, position={above=1mm}](1)(5);
                \Edge [style=dashed, label=$K$, position={below right=1mm}](3)(4);
            \end{tikzpicture}
            \caption{Unicyclic Subgraph}
            \label{Subgraph}
        \end{figure}
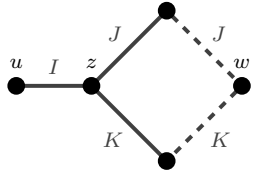

Next we will give a contrapositive proof of the forward direction.  Let $G$ be a graph that is neither a tree nor a cycle. Let $R$ be a ring with ideals $I,J,K$  such that $I+(J\cap K)\neq (I+J)\cap(I+K)$.  Since $G$ is not a tree, there exists a cycle $C$ within $G$, and since $G$ is not a cycle, there exists a vertex $u$ not in $C$ that is adjacent to some vertex $z$ in $C$. Define a subgraph $G'$ of $G$ with $V(G')=V(C)\cup\{u\}$ and $E(G')=E(C)\cup\{uz\}$ as in Figure \ref{Subgraph}. Let $w$ be a vertex in $C$ that is distinct from $z$. Then, define $\alpha':E(G')\rightarrow{I(R)}$ such that $\alpha'(uz)=I$, for any edge $e$ in the clockwise direction from $z$ to $w$, $\alpha'(e)=J$, and for any edge $e'$ in the counterclockwise direction from $z$ to $w$, $\alpha'(e')=K$. Note that since $C$ is a cycle and the edge $uz$ along with vertices $u$ and $z$ defines a path, we can apply Theorem \ref{paste} to $(G', \alpha')$ to determine whether UDP holds on $(G', \alpha')$. We have $\bigcap_{P\in\mathcal{P}_{(u,w)}}\alpha'(P)=(I+K)\cap(I+J)$ and $(\bigcap_{P\in\mathcal{P}_{(u,z)}}\alpha'(P))+(\bigcap_{P\in\mathcal{P}_{(z,w)}}\alpha'(P))=I+(J\cap{K})$, but $(I+J)\cap(I+K)\neq{I+(J\cap{K})}$, so UDP does not hold on $(G', \alpha')$. Thus, there exists an edge-labeling $\alpha'$ on a subgraph $G'$ of $G$ such that UDP does not hold on $(G',\alpha')$, and it follows from the contrapositive of Lemma \ref{Every} that there exists an edge-labeling on $G$ such that UDP does not hold.
\end{proof}

We now prove a theorem that will allow us to check only edge-labelings over a particular unicyclic graph rather than edge-labelings over every possible graph.

\begin{thm}\label{latest from Remi}
Let $R$ be an integral domain.  If there exists a graph $G$ that is neither a tree nor a cycle such that every edge-labeling of $G$ over $R$ satisfies UDP, then $R$ is a Prüfer domain.
\end{thm}

\begin{proof}
We give a contrapositive proof.  Suppose that $R$ is an integral domain that is not a Prüfer domain, so there exist ideals $I,J,K$ of $R$ such that $I+(J\cap K)\neq (I+J)\cap (I+K)$.  Then the unicyclic graph $H$ in Figure \ref{fig:4-vertex unicyclic graph} is an edge-labeled graph over $R$ on which UDP does not hold.  Let $G$ be a graph that is neither a tree nor a cycle.  Then $G$ is a supergraph of a subdivision of the graph $H$ in Figure \ref{fig:4-vertex unicyclic graph}.  Since $H$ does not satisfy UDP, we know from Theorem \ref{extension} that no subdivision of $H$ satisfies UDP.  Then by Lemma \ref{Every} we know that there is some edge-labeling $\alpha$ of $G$ over $R$ for which $(G,\alpha)$ does not satisfy UDP.
\end{proof}

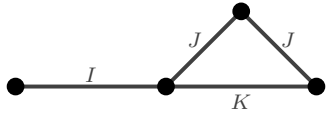
\begin{figure}[h! bt]
    \centering
    \begin{tikzpicture}
    \Vertex[size=0.2, color=black]{1}
    \Vertex[size=0.2, color=black, x=2]{2}
    \Vertex[size=0.2, color=black, x=4]{3}
    \Vertex[size=0.2, color=black, x=3, y=1]{4}
    \Edge[label=$I$, position=above](1)(2)
    \Edge[label=$K$, position=below](2)(3)
    \Edge[label=$J$, position=above left=1mm](2)(4)
    \Edge[label=$J$, position=above right=1mm](3)(4)
    \end{tikzpicture}
    \caption{An edge-labeled unicylcic graph $H$}
    \label{fig:4-vertex unicyclic graph}
\end{figure}

\begin{cor}{result of Remi extra work}
Let $R$ be an integral domain.  The following are equivalent:
\begin{enumerate}
\item\label{Prüfer domain} $R$ is a Prüfer domain.
\item \label{every graph} Every edge-labeled graph over $R$ satisfies UDP.
\item \label{theorem above} There exists a graph $G$ that is neither a tree nor a cycle such that every edge-labeling of $G$ over $R$ satisfies UDP.
\end{enumerate}
\end{cor}

\begin{proof}
Let $R$ be an integral domain.  The equivalence of (\ref{Prüfer domain}) and (\ref{every graph}) was shown in \cite{UDP} and is stated here as Theorem \ref{intersection_over_sum}.  Note that (\ref{every graph}) implies (\ref{theorem above}) is trivially true, and (\ref{theorem above}) implies (\ref{Prüfer domain}) is Theorem \ref{latest from Remi}.
\end{proof}

From this corollary, it follows that when working over an integral domain $R$, one can verify that UPD holds for every edge-labeled graph over $R$ by verifying that UDP holds over every possible edge-labeling of the unicyclic graph in Figure \ref{fig:4-vertex unicyclic graph}.


\section*{Declaration of competing interest}

The authors declare that they have no known competing financial interests or personal relationships that could have appeared to influence the work reported in this paper.

\section*{Data availability}

No data was used for the research described in this article.

\section*{Acknowledgements}

The authors acknowledge support from the University of Texas at Tyler Research Experience for Undergraduates, NSF Grant DMS-2149921, in which the second, third, fourth, and fifth authors participated during Summer 2024, under the mentorship of the first author.




\end{document}